\documentclass[11pt,reqno]{amsart}

\usepackage[utf8]{inputenc}
\usepackage{amsmath,amssymb,amsthm,mathtools}
\usepackage{xcolor}
\usepackage[colorlinks=true,linkcolor=blue,citecolor=blue,urlcolor=blue]{hyperref}
\usepackage[capitalize]{cleveref}

\usepackage{aliascnt}
\theoremstyle{plain}
\newtheorem{theorem}{Theorem}[section]
\newaliascnt{lemma}{theorem}
\newtheorem{lemma}[lemma]{Lemma}
\aliascntresetthe{lemma}
\newaliascnt{proposition}{theorem}
\newtheorem{proposition}[proposition]{Proposition}
\aliascntresetthe{proposition}
\newaliascnt{corollary}{theorem}
\newtheorem{corollary}[corollary]{Corollary}
\aliascntresetthe{corollary}
\theoremstyle{remark}
\newaliascnt{remark}{theorem}
\newtheorem{remark}[remark]{Remark}
\aliascntresetthe{remark}

\numberwithin{equation}{section}

\newcommand{\R}{\mathbb{R}}

\title[A compactly supported real scalar Hardy endpoint]
      {Compactly supported real scalar potentials realizing the Hardy uncertainty
       endpoint for Schr\"odinger evolutions}

\author[Xiao-Ming Fu]{Xiao-Ming Fu}
\address{School of Mathematical Sciences, University of Science and Technology of China,
Hefei, Anhui 230026, P.R.\ China}
\email{fuxm@ustc.edu.cn}
\author[Tianyang Sun]{Tianyang Sun}
\address{School of Mathematical Sciences, University of Science and Technology of China,
Hefei, Anhui 230026, P.R.\ China}
\email{tysun@mail.ustc.edu.cn}

\date{July 2026}

\begin{document}

\begin{abstract}
At the critical Hardy Gaussian weight for the Schr\"odinger equation in one space
dimension on $[0,1]$, the known nonzero scalar example in the weighted
$L^2$ class carries a complex-valued potential. Cassano and Fanelli observed that the
existence of a \emph{real-valued} scalar endpoint example was open, and produced examples
with real electric and magnetic potentials only after introducing a magnetic potential.
We construct a nonzero smooth solution of
$i\partial_t u+\partial_x^2 u=Vu$ on $\R\times[0,1]$ such that
$e^{x^2/4}u(\cdot,0),e^{x^2/4}u(\cdot,1)\in L^2(\R)$ and the potential is
\emph{bounded, smooth, real-valued, and purely scalar} and is supported in one fixed compact
spatial interval for all times. This compact-support endpoint example is the main result.
We also record the explicit rational-tail realization underlying the construction.
The main results of this paper were obtained by the multi-agent system Eureka and have
subsequently been verified by the authors.
\end{abstract}

\maketitle

\section{Introduction and statement of results}\label{sec:intro}

Consider the one-dimensional Schr\"odinger evolution on the normalized time interval
$[0,1]$,
\begin{equation}\label{eq:schr}
  i\partial_t u+\partial_x^2u = V(x,t)\,u
  \qquad\text{on }\R\times[0,1],
\end{equation}
with a bounded potential $V\in L^\infty(\R\times[0,1])$. Hardy's uncertainty
principle \cite{Hardy-1933} states, in its dynamical reformulation, that Gaussian
decay of $u$ at two distinct times forces $u\equiv0$ once the decay rates are strong
enough. The sharp nonfree quantitative form was obtained in
\cite{EKPV-Duke-2010}, building on the free and logarithmic-convexity results
\cite{EKPV-free-2008,EKPV-JEMS-2008} (see also
\cite{Cowling-2010,EKPV-BAMS-2012} for surveys and refinements): if $u$ solves
\eqref{eq:schr} and the bounded potential $V$ satisfies suitable spatial-decay or
structural assumptions, and, for some $\alpha,\beta>0$,
\begin{equation}\label{eq:weights}
  e^{x^2/\beta^2}u(\cdot,0)\in L^2(\R),
  \qquad
  e^{x^2/\alpha^2}u(\cdot,1)\in L^2(\R),
\end{equation}
then
\begin{equation}\label{eq:sharp}
  \alpha\beta<4 \;\Longrightarrow\; u\equiv 0 .
\end{equation}
The constant $4$ is sharp: at $\alpha\beta=4$, Escauriaza, Kenig, Ponce and Vega
constructed a bounded complex-valued potential admitting a nonzero smooth solution
satisfying \eqref{eq:weights} \cite[Thm.~2]{EKPV-Duke-2010}. Uniformly for
$t\in[0,1]$, their potential obeys
\[
  |V(x,t)|\lesssim \frac{1}{1+|x|^2}.
\]

The earlier logarithmic-convexity uniqueness result was established in
\cite{EKPV-JEMS-2008}, while the sharp threshold and the complex-valued
endpoint construction were obtained in \cite{EKPV-Duke-2010}. The existence
of an analogous real-valued endpoint example was explicitly posed as an open
problem in the later survey \cite{EKPV-BAMS-2012}. Under certain structural
assumptions, Cassano and Fanelli proved the sharp threshold \eqref{eq:sharp} for
electromagnetic potentials \cite{CassanoFanelli-2015}. By introducing a real magnetic
potential $\mathbf A$, they constructed an endpoint example in which both the electric
and magnetic potentials are real-valued. Their computation also shows directly why this
is not already a real scalar example: for their explicit profile (in their notation),
\[
  \operatorname{Im}\frac{(i\partial_t+\Delta)u}{u}
  =-\frac{2kt}{(1+t^2)(1+r^2)}\not\equiv0,
\]
and the term $-i\,\operatorname{div}\mathbf A$ in the magnetic Laplacian cancels precisely
this imaginary part \cite[Sec.~2]{CassanoFanelli-2015}. Since their profile never vanishes,
the same solution cannot satisfy a purely scalar equation with a real potential. Moreover,
their magnetic field satisfies $\operatorname{curl}\mathbf A\not\equiv0$, so $\mathbf A$
cannot be removed by a gauge transformation. The standard Euclidean, scaling, and Appell
transformations preserve whether the magnetic field vanishes identically and therefore do
not transform their construction into the purely scalar setting. See also
\cite{Barcelo-2013} for non-sharp covariant Hardy inequalities in the magnetic setting.

The purpose of this note is to give a nonzero endpoint solution in the purely scalar case
with a real potential. More strongly, the potential can be supported in a fixed compact
spatial interval for all $t\in[0,1]$.
By the standard scaling and Appell transformations, general endpoint parameters
$T>0$ and $\alpha\beta=4T$ reduce to $T=1$ and $\alpha=\beta=2$. We therefore state and
prove the result in this symmetric normalization.

Conceptually, this example shows that none of the following is necessary for endpoint
nonuniqueness: gain or dissipation from an imaginary part of the potential, the additional
freedom of a magnetic field or gauge structure, or a slowly decaying long-range tail at
spatial infinity. It thereby isolates the interaction between Schr\"odinger evolution and
the critical Gaussian weight itself.

\begin{theorem}\label{thm:main}
Equation \eqref{eq:schr} admits a nonzero smooth solution $u$ on $\R\times[0,1]$ with a
potential
\[
  V\in C^\infty(\R\times[0,1];\R)\cap L^\infty(\R\times[0,1])
\]
that is real-valued, purely scalar, and supported in a fixed compact spatial interval
for all $t\in[0,1]$,
such that
\[
  e^{x^2/4}u(\cdot,0)\in L^2(\R),
  \qquad
  e^{x^2/4}u(\cdot,1)\in L^2(\R).
\]
\end{theorem}

The construction transports a fixed positive profile by a pseudoconformal (lens)
quadratic phase. The potential is recovered by inverse design from the continuity equation,
which forces its imaginary part to vanish identically. A suitable profile makes the
resulting potential vanish outside a compact interval. We also give an explicit elementary
rational-tail profile whose potential is available in closed form and is $O(|x|^{-2})$.

The construction is elementary and self-contained, using only the pseudoconformal
symmetry of the free Schr\"odinger equation and a one-line ordinary differential
equation for the Gaussian width. The remainder of the note is organized as follows.
\Cref{sec:construction} proves \cref{thm:main}. \Cref{sec:remarks} comments on the
relation to \cite{CassanoFanelli-2015}.

\section{The construction}\label{sec:construction}

Throughout this section the space dimension is one and the time interval is $[0,1]$.
We use the symmetric endpoint widths $A=B=2$, and write $s=|\xi|$ and $r=|x|$.

\subsection{The transported profile}
Define the Gaussian width
\begin{equation}\label{eq:y}
  y(t)=2\bigl((1-t)^2+t^2\bigr)^{1/2},
  \qquad t\in[0,1].
\end{equation}
Since $4(1-t)^2+4t^2$ is strictly positive on $[0,1]$, the function $y$ is
smooth and bounded below by a positive constant. The following lemma records the two
properties of $y$ that drive the construction.

\begin{lemma}\label{lem:y}
The function $y$ in \eqref{eq:y} satisfies $y(0)=y(1)=2$, and the ordinary
differential equation
\begin{equation}\label{eq:ode}
  y''=\frac{16}{y^{3}} \qquad\text{on }[0,1].
\end{equation}
\end{lemma}

\begin{proof}
The endpoint values are immediate from \eqref{eq:y}. Write
$q=y^2=4-8t+8t^2$.
For any quadratic $q=at^2+bt+c$ one has the identity $2q''q-(q')^2=4ac-b^2$, a constant.
Here
\[
  4ac-b^2=4\cdot8\cdot4-(-8)^2=64.
\]
Since $y=q^{1/2}$ gives
$y''=\bigl(2q''q-(q')^2\bigr)/(4q^{3/2})=\bigl(2q''q-(q')^2\bigr)/(4y^3)$, we obtain
$y''=64/(4y^3)=16/y^3$.
\end{proof}

Set
\begin{equation}\label{eq:b}
  b(t)=\frac{y'(t)}{4y(t)},
\end{equation}
and let $F\colon\R\to\R$ be a smooth profile (to be specified). Define the transported
ansatz
\begin{equation}\label{eq:ansatz}
  u(x,t)=y(t)^{-1/2}\,F\!\bigl(x/y(t)\bigr)\,
  \exp\!\bigl(i\,b(t)x^2+i\,\theta(t)\bigr),
\end{equation}
with a real phase $\theta$ to be chosen. Write $u=\rho\,e^{i\phi}$ with
\[
  \rho(x,t)=y^{-1/2}F(x/y),
  \qquad
  \phi(x,t)=b(t)x^2+\theta(t).
\]

\begin{lemma}[Reality of the potential]\label{lem:real}
For any real profile $F$ and with $b=y'/(4y)$ as in \eqref{eq:b}, the amplitude and
phase satisfy the continuity equation
\begin{equation}\label{eq:cont}
  \rho_t+2\rho_x\phi_x+\rho\,\phi_{xx}=0 .
\end{equation}
Consequently, wherever $u\ne0$,
\[
  \frac{i\partial_t u+\partial_x^2u}{u}
  = -\phi_t-|\phi_x|^2+\frac{\rho_{xx}}{\rho}
  \;\in\;\R,
\]
so that the potential $V:=-\phi_t-|\phi_x|^2+\rho_{xx}/\rho$ defined by
\eqref{eq:schr} is real-valued.
\end{lemma}

\begin{proof}
Writing $u=\rho e^{i\phi}$ and separating real and imaginary parts,
\[
  i\partial_t u+\partial_x^2u
  =\Bigl[\bigl(-\rho\phi_t+\rho_{xx}-\rho|\phi_x|^2\bigr)
  +i\bigl(\rho_t+2\rho_x\phi_x+\rho\phi_{xx}\bigr)\Bigr]e^{i\phi}.
\]
Hence $(i\partial_tu+\partial_x^2u)/u$ is real precisely when the imaginary bracket
\eqref{eq:cont} vanishes. To check \eqref{eq:cont}, put $\xi=x/y$. From
$\rho=y^{-1/2}F(\xi)$ one computes
\[
  \rho_t=-\frac12\frac{y'}{y}\rho-\frac{y'}{y}\,y^{-1/2}\,\xi F'(\xi),
  \qquad
  \rho_x=y^{-3/2}F'(\xi).
\]
With $\phi=bx^2+\theta$ we have $\phi_x=2bx$ and $\phi_{xx}=2b$, so
\[
  2\rho_x\phi_x=4b\,y^{-1/2}\,\xi F'(\xi),
  \qquad
  \rho\,\phi_{xx}=2b\,y^{-1/2}F(\xi).
\]
Summing,
\[
  \rho_t+2\rho_x\phi_x+\rho\phi_{xx}
  =y^{-1/2}\Bigl[\Bigl(4b-\tfrac{y'}{y}\Bigr)\xi F'
  +\Bigl(2b-\tfrac12\tfrac{y'}{y}\Bigr)F\Bigr].
\]
Both coefficients vanish because $b=y'/(4y)$, i.e.\ $4b=y'/y$. This proves
\eqref{eq:cont}, and the displayed formula for $V$ follows.
\end{proof}

\Cref{lem:real} makes $V$ automatically real for \emph{any} real profile. The remaining
freedom in $F$ and $\theta$ is used to make $V$ bounded.

\subsection{Cancellation of the quadratic term}
Expanding $V=-\phi_t-|\phi_x|^2+\rho_{xx}/\rho$ with $\phi=bx^2+\theta$ gives,
using $\rho_{xx}/\rho=y^{-2}(F''/F)(\xi)$,
\begin{equation}\label{eq:Vexpand}
  V(x,t)=-\bigl(b'+4b^2\bigr)x^2-\theta'
  +\frac1{y^2}\,\frac{F''}{F}\Bigl(\frac{x}{y}\Bigr).
\end{equation}
A direct computation from $b=y'/(4y)$ yields $b'+4b^2=y''/(4y)$. Meanwhile, for the even
profile $F=e^{g(s)}$ (with $s=|\xi|$) one has, with smooth extension at $s=0$,
\begin{equation}\label{eq:LapF}
  \frac{F''}{F}=g''+(g')^2 .
\end{equation}
We now choose
\begin{equation}\label{eq:profileF}
  F_N(\xi)=e^{-\xi^2}(1+\xi^2)^{-k}=e^{g(s)},
  \qquad g(s)=-s^2-k\log(1+s^2),
\end{equation}
for a parameter $k>0$ specified below, and set $p(s)=-k\log(1+s^2)$, so $g=-s^2+p$.
Substituting into \eqref{eq:LapF},
\begin{equation}\label{eq:LapFN}
  \frac{F_N''}{F_N}=4s^2-2+R_N(s),
  \qquad
  R_N(s)=p''+(p')^2-4sp'.
\end{equation}
For $p(s)=-k\log(1+s^2)$ one finds the closed form
\begin{equation}\label{eq:RN}
  R_N(s)=\frac{2k\bigl(4s^4+(2k+5)s^2-1\bigr)}{(1+s^2)^2},
  \qquad
  R_{N,\infty}:=\lim_{s\to\infty}R_N(s)=8k,
\end{equation}
and moreover $R_N(s)-8k=(4k^2-6k)s^{-2}+O(s^{-3})$ as $s\to\infty$, so
$R_N$ is bounded and smooth on $[0,\infty)$ with $R_N(s)-R_{N,\infty}=O(s^{-2})$.

Insert \eqref{eq:LapFN} into \eqref{eq:Vexpand}. The quadratic term $4s^2/y^2=4x^2/y^4$
coming from $F''/F$ combines with $-(b'+4b^2)x^2=-\tfrac{y''}{4y}x^2$ to give the
coefficient
\[
  -\frac{y''}{4y}+\frac{4}{y^4}\quad\text{of }x^2,
\]
which vanishes precisely by the ODE \eqref{eq:ode}. Therefore, choosing
\begin{equation}\label{eq:theta}
  \theta'(t)=\frac{-2+R_{N,\infty}}{y(t)^2}=\frac{8k-2}{y(t)^2},
\end{equation}
we are left with
\begin{equation}\label{eq:VN}
  V_N(x,t)=\frac{R_N(|x|/y(t))-R_{N,\infty}}{y(t)^2}
  =\frac{2k\bigl((2k-3)x^2-5y(t)^2\bigr)}{\bigl(x^2+y(t)^2\bigr)^2} .
\end{equation}
The second equality follows by substituting $s=|x|/y(t)$ into \eqref{eq:RN}, and
exhibits $V_N$ as a manifestly real and smooth function.
By \eqref{eq:RN} this $V_N$ is real, smooth, bounded, and $O(|x|^{-2})$ as
$|x|\to\infty$ uniformly in $t\in[0,1]$ (the function $y$ is continuous and strictly
positive on the compact interval $[0,1]$, hence bounded above and below by positive
constants; thus $s=|x|/y(t)\to\infty$ uniformly in $t$ as $|x|\to\infty$, and since
$|x|^2=s^2y^2$ the closed form decays like $|x|^{-2}$).

\subsection{The odd profile}
For the odd (Dirichlet-parity) example take
\begin{equation}\label{eq:profileFD}
  F_D(\xi)=\xi\,F_N(\xi)=\xi\,e^{-\xi^2}(1+\xi^2)^{-k}.
\end{equation}
Since $(\xi h)''=\xi(h''+2h'/s)$ for even $h=h(s)$, one gets, across the
smooth extension over $\{\xi=0\}$,
\begin{equation}\label{eq:LapFD}
  \frac{F_D''}{F_D}=\frac{F_N''}{F_N}+\frac{2g'}{s}
  =4s^2-2+R_D(s),
  \qquad
  R_D(s)=R_N(s)-4-\frac{4k}{1+s^2},
\end{equation}
with $R_{D,\infty}=8k-4$. The same computation as above, with $R_N$ replaced by $R_D$
and $\theta'$ chosen accordingly, produces a real, smooth, bounded potential
$V_D=(R_D(|x|/y)-R_{D,\infty})/y^2=O(|x|^{-2})$; in closed form,
\begin{equation}\label{eq:VD}
  V_D(x,t)=\frac{2k\bigl((2k-5)x^2-7y(t)^2\bigr)}{\bigl(x^2+y(t)^2\bigr)^2}.
\end{equation}
The corresponding solution
$u_D=y^{-1/2}F_D(x/y)\,e^{i(bx^2+\theta)}$ is smooth and odd in $x$, and the potential
$V_D$ is smooth and bounded across the zero set $\{x=0\}$ because it depends only on
$|x|/y$.

\subsection{Endpoint integrability and conclusion}
For the rational-tail profiles, it remains to verify the weighted $L^2$ conditions. At
$t=0$, $y(0)=2$, so from
\eqref{eq:ansatz} and \eqref{eq:profileF},
\[
  e^{x^2/4}u_N(x,0)=2^{-1/2}(1+x^2/4)^{-k}e^{i\phi(x,0)},
\]
because the weight $e^{x^2/4}$ cancels the Gaussian $e^{-x^2/4}$ in $F_N$. Hence
\[
  \int_{\R}e^{x^2/2}|u_N(x,0)|^2\,dx
  =\frac12\int_{\R}(1+x^2/4)^{-2k}\,dx<\infty
  \iff 4k>1 .
\]
The same computation at $t=1$ gives the terminal bound under the same condition. For the
odd profile the extra factor $x/2$ contributes two more powers of $r$, so $4k>3$ suffices.
Choosing $k>3/4$ makes both the even and odd examples
admissible.

\begin{proposition}[Rational-tail example]\label{prop:rational}
Equation \eqref{eq:schr} admits nonzero smooth even and odd solutions. They satisfy the
endpoint bounds of \cref{thm:main}, and their bounded real scalar potentials are
$O(|x|^{-2})$ uniformly in $t\in[0,1]$.
\end{proposition}

\begin{proof}
Fix $k>3/4$. Define $y$ by \eqref{eq:y}, $b$ by \eqref{eq:b}, $\theta$ by
\eqref{eq:theta} (respectively its $R_D$-analogue), and $u_\sigma$ by \eqref{eq:ansatz}
with $F_\sigma$, $\sigma\in\{N,D\}$. By \cref{lem:real} the potential is real, and by the
cancellation of the quadratic term (via \cref{lem:y}) it equals the bounded
$O(|x|^{-2})$ function \eqref{eq:VN} (resp.\ its $R_D$-analogue); a direct substitution
confirms that $u_\sigma$ solves \eqref{eq:schr} with this $V_\sigma$. The weighted
endpoint bounds hold by the integrability computation above. Finally $u_N$ is even and
$u_D$ is odd in $x$.
\end{proof}

\subsection{A master construction}\label{subsec:master}
The reality and boundedness of the potential depend on the profile only through the
zeroth-order remainder left after removing the Gaussian part of $F''$, so the construction
proceeds from two conditions on $F$, of which the rational tail is one convenient
realization. We phrase the first condition as a pointwise identity so that the resulting
Schr\"odinger identity remains valid even at zeros of $F$.

\begin{proposition}\label{prop:master}
Let $F\in C^\infty(\R;\R)$ with $F\not\equiv0$, and suppose that
\begin{enumerate}
\item[\textup{(i)}] there is a bounded smooth function $W_F\colon\R\to\R$, with a finite
limit $W_\infty:=\lim_{|\xi|\to\infty}W_F(\xi)$, such that
\[
  F''(\xi)=\bigl(4\xi^2+W_F(\xi)\bigr)F(\xi)\quad\text{on }\R;
\]
\item[\textup{(ii)}] $e^{|\cdot|^2}F\in L^2(\R)$.
\end{enumerate}
Let $y,b$ be as in \eqref{eq:y} and \eqref{eq:b}, and take
$\theta'(t)=W_\infty/y(t)^2$. Then
$u=y^{-1/2}F(x/y)\,e^{i(bx^2+\theta)}$ is a nonzero smooth solution of \eqref{eq:schr} on
$\R\times[0,1]$ whose potential
\begin{equation}\label{eq:Vmaster}
  V(x,t)=\frac{W_F(x/y(t))-W_\infty}{y(t)^2}
\end{equation}
is smooth, real, scalar and bounded, and $u$ satisfies the endpoint bounds
$e^{x^2/4}u(\cdot,0),e^{x^2/4}u(\cdot,1)\in L^2(\R)$. If in addition
$W_F(\xi)-W_\infty=O(|\xi|^{-2})$, then $V=O(|x|^{-2})$ uniformly in $t$.
\end{proposition}

\begin{proof}
Write $u=\rho e^{i\phi}$ with $\rho=y^{-1/2}F(x/y)$ and $\phi=bx^2+\theta$. Separating
real and imaginary parts --- an identity valid on all of $\R\times[0,1]$, cf.\
\cref{lem:real} ---
\[
  i\partial_t u+\partial_x^2u=\bigl[(\rho_{xx}-\rho\phi_t-\rho|\phi_x|^2)
  +i(\rho_t+2\rho_x\phi_x+\rho\phi_{xx})\bigr]e^{i\phi}.
\]
The imaginary bracket vanishes identically by the continuity equation \eqref{eq:cont}, which
holds for every real profile. For the real bracket, the identity (i) gives
\[
\begin{aligned}
  \rho_{xx}=y^{-5/2}F''(x/y)
  &=y^{-5/2}\bigl(4x^2/y^2+W_F(x/y)\bigr)F(x/y)\\
  &=\Bigl(\frac{4x^2}{y^4}+\frac{W_F(x/y)}{y^2}\Bigr)\rho .
\end{aligned}
\]
Since $\phi_t+|\phi_x|^2=(b'+4b^2)x^2+\theta'=\tfrac{4}{y^4}x^2+\theta'$ by
$b'+4b^2=y''/(4y)$ and the ODE \eqref{eq:ode}, the $x^2$ terms cancel and, with
$\theta'=W_\infty/y^2$,
\[
  \rho_{xx}-\rho(\phi_t+|\phi_x|^2)=\frac{W_F(x/y)-W_\infty}{y^2}\,\rho=V\rho .
\]
Hence $i\partial_t u+\partial_x^2u=Vu$ on all of $\R\times[0,1]$, including the zero set
of $u$, with $V$ as in \eqref{eq:Vmaster}. By (i) it is smooth, real and bounded, and it is
$O(|x|^{-2})$ when $W_F-W_\infty=O(|\xi|^{-2})$, since $y$ is bounded above and below on
$[0,1]$. For the endpoint bounds, at both $t=0$ and $t=1$ one has $y=2$, and the
substitution $\xi=x/2$ gives
\[
  \int_{\R}e^{x^2/2}|u(x,t)|^2\,dx
  =\int_{\R}e^{2\xi^2}|F(\xi)|^2\,d\xi<\infty,
  \qquad t\in\{0,1\},
\]
by (ii).
\end{proof}

The rational-tail profiles used above are instances of this proposition. More generally,
for an even positive $h$ and $F(\xi)=e^{-\xi^2}h(\xi)$ one has
\[
  W_F=-2+\frac{h''}{h}-4\xi\,\frac{h'}{h}.
\]
Thus (i) holds when $h''/h$ and $\xi h'/h$ are bounded and have finite limits. This includes
$h=(1+\xi^2)^{-k}$ and the log-modified tails
$h=(1+\xi^2)^{-k}(\log(e+\xi^2))^{-m}$ whenever (ii) holds: $4k>1$, or $4k=1$ with
$m>\tfrac12$. Since the support of $V$ is controlled by that of $W_F-W_\infty$, choosing
the latter to vanish outside a compact interval produces the main result.

\begin{corollary}\label{cor:nonzero}
Every profile satisfying \textup{(i)--(ii)} yields a potential \eqref{eq:Vmaster} that is not
identically zero.
\end{corollary}

\begin{proof}
If $V\equiv0$ then $W_F\equiv W_\infty$, so (i) reads $HF=-W_\infty F$ for the
one-dimensional harmonic oscillator $H=-\partial_\xi^2+4\xi^2$. By (ii), $F\in L^2(\R)$.
Pairing this distributional identity with each Hermite eigenfunction $\psi_j$ of $H$ gives
$(\lambda_j+W_\infty)\langle F,\psi_j\rangle=0$. Since the Hermite functions form an
orthonormal basis of $L^2(\R)$, a nonzero $F$ lies in one eigenspace and has the form
$F=e^{-\xi^2}P$ with $P$ a nonzero Hermite polynomial. Then
$e^{\xi^2}F=P\notin L^2(\R)$, contradicting (ii). The pure Gaussian $F=e^{-\xi^2}$ is the
ground-state case $P\equiv1$, for which $V\equiv0$ but condition (ii) fails.
\end{proof}

\begin{proof}[Proof of \cref{thm:main}]
Fix $R>0$ and $\nu>\tfrac14$. On $r>R$ let
\[
  h(r)=U\!\bigl(\nu,\tfrac12,2r^2\bigr),
\]
where $U$ is the Tricomi confluent hypergeometric function
\cite[\S\S13.2, 13.4, 13.7]{Olver-DLMF}. For $\nu>0$ and $z>0$, its integral representation
\[
  U(\nu,b,z)=\frac1{\Gamma(\nu)}\int_0^\infty e^{-zt}t^{\nu-1}(1+t)^{b-\nu-1}\,dt
\]
shows that it is smooth and positive and that $U(\nu,b,z)\sim z^{-\nu}$ as
$z\to\infty$. Writing $f(z)=U(\nu,\tfrac12,z)$, Kummer's equation
$zf''+(\tfrac12-z)f'-\nu f=0$ becomes
\[
  h''(r)-4r h'(r)=8\nu h(r),
\]
so the corresponding profile remainder equals the constant $8\nu-2$ for $r>R$.

Extend $\log h$ smoothly to $[0,\infty)$, constant near $r=0$ and unchanged for
$r\ge R$. After exponentiating and reflecting evenly, let the same letter denote the
resulting positive smooth function on $\R$. Set $F(\xi)=e^{-\xi^2}h(\xi)$ and
\[
  W_F(\xi)=\frac{F''(\xi)}{F(\xi)}-4\xi^2.
\]
Then $F$ and $W_F$ are smooth, $F>0$, and $W_F\equiv W_\infty:=8\nu-2$ for
$|\xi|>R$. Moreover,
$e^{\xi^2}F(\xi)=h(\xi)=O(|\xi|^{-2\nu})$, which belongs to $L^2(\R)$ because
$4\nu>1$. Proposition~\ref{prop:master} therefore gives a nonzero smooth endpoint solution
with a smooth bounded real scalar potential. Formula \eqref{eq:Vmaster} shows that
$V(x,t)=0$ whenever $|x|>R\,y(t)$. Since $y(t)\leq 2$ on $[0,1]$, it follows that
\[
  \operatorname{supp}V(\cdot,t)\subseteq[-2R,2R]
  \qquad\text{for every }t\in[0,1].
\]
Thus the potential is supported in one fixed compact spatial interval for all times.
This proves \cref{thm:main}.
\end{proof}

\section{Remarks}\label{sec:remarks}

\begin{remark}[Higher dimensions and the magnetic example of \cite{CassanoFanelli-2015}]
The main result is stated in one dimension and for $T=1$. The mechanism itself is
dimension-independent: on $\mathbb R^n$ one replaces $y^{-1/2}$ by $y^{-n/2}$ and
$\partial_x^2$ by $\Delta$, while the compact-support profile is obtained from
$U(\nu,n/2,2r^2)$ with $\nu>n/4$. The same computation therefore gives the analogous
endpoint example in every dimension $n\ge1$; general time intervals follow by taking
endpoint widths $A,B>0$ with $AB=4T$.

Cassano and Fanelli obtain an endpoint example with real electric and magnetic
potentials by adjoining a magnetic potential
$\mathbf A\not\equiv0$; the imaginary zero-order term $i\,\mathrm{div}\,\mathbf A$ is what allows the
Gaussian profile to persist at the endpoint. Our construction is purely scalar
($\mathbf A\equiv0$): reality of $V$ is enforced instead by the continuity equation
\eqref{eq:cont}, which holds identically once the phase is the pseudoconformal one,
$b=y'/(4y)$. By contrast, the magnetic endpoint construction of
\cite{CassanoFanelli-2015} is three-dimensional.
To the best of our knowledge no purely scalar real-valued endpoint example was available
previously. In the free case the weighted $L^2$ endpoint is trivial; the previously known
nonzero endpoint examples are the complex-valued scalar example of \cite[Thm.~2]{EKPV-Duke-2010}
and the real \emph{electromagnetic} example of \cite{CassanoFanelli-2015}. We do not treat
discrete Schr\"odinger operators or non-Gaussian (Morgan-type) weights.
\end{remark}

\begin{remark}[The role of the rational tail]
The extremal profile in the pointwise free Hardy theorem is a pure Gaussian, which in our
construction corresponds to $k=0$: then $R_N\equiv0$ by \eqref{eq:RN}, so $V_N\equiv0$ and
$u$ is the free lens solution. Thus a pure Gaussian amplitude \emph{is} sustained at
the endpoint by a bounded (indeed vanishing) scalar potential; the obstruction is not the
boundedness of $V$ but integrability. At $k=0$ the critical Gaussian weight cancels the
amplitude exactly, leaving the non-integrable constant $2^{-1/2}|e^{i\phi}|$, so the pure
Gaussian fails the weighted $L^2$ endpoint condition, which requires $4k>1$. The rational
tail $(1+\xi^2)^{-k}$ plays a double role: it leaves an integrable polynomial factor once
the critical weight cancels, so that the strict weighted $L^2$ endpoint condition holds,
and it simultaneously renders $V$ a nonzero but bounded $O(|x|^{-2})$ real scalar
potential. The construction thus lives at the endpoint of the Hardy \emph{inequality}
without contradicting the rigidity of its free extremals, precisely because the weighted
$L^2$ endpoint excludes the pure Gaussian.
\end{remark}

\section*{Acknowledgements}
The main results of this paper were obtained by Eureka, a multi-agent system for resolving
mathematical conjectures through human--AI interaction, and were subsequently verified by
the authors. The proof of Theorem~\ref{thm:main} has additionally been formally verified
in the Lean proof assistant; the formalization is available in the
\href{https://github.com/Shealizon/hardy-compact-support-counterexample}{project repository}.
This work is supported by the Fundamental and Interdisciplinary Disciplines Breakthrough Plan
of the Ministry of Education of China.


\end{document}